\documentclass[11pt]{amsart}
\usepackage{amssymb}
\usepackage{color}
\usepackage{graphicx}
\usepackage{float}
\usepackage[all,cmtip]{xy}
\usepackage{tikz}
\usetikzlibrary{matrix}
\usepackage{url}
\usepackage{subfig}
\usepackage{hyperref}
\newcommand{\Maps}{\operatorname{\mathcal{M}}\mathfrak{aps}}

\mathchardef\mhyphen="2D

\newcommand{\std}{{\operatorname{std}}}

\newcommand{\Diff}{\operatorname{Diff}}

\newcommand{\SO}{\operatorname{SO}}
\newcommand{\U}{\operatorname{U}}
\newcommand{\SU}{\operatorname{SU}}

\newcommand{\im}{{\operatorname{Image}}}
\newcommand{\Id}{{\operatorname{Id}}}

\newcommand{\Rot}{\operatorname{Rot}}
\newcommand{\tb}{\operatorname{tb}}

\newcommand{\Top}{\operatorname{Top}}
\newcommand{\Cont}{\operatorname{Cont}}

\newcommand{\R}{{\mathbb{R}}}
\newcommand{\C}{{\mathbb{C}}}
\newcommand{\Z}{{\mathbb{Z}}}

\newcommand{\NS}{{\mathbb{S}}}
\newcommand{\D}{{\mathbb{D}}}

\newcommand{\F}{{\mathcal{F}}}

\newcommand{\Leg}{{\mathfrak{Leg}}}

\newcommand{\FLeg}{{\mathfrak{FLeg}}}

\newcommand{\FLegO}{{\overline{\mathfrak{FLeg}}}}

\newcommand{\Knots}{{\widehat{\mathcal{K}}}}

\newcommand{\Imm}{{\mathfrak{Imm}}}
\newcommand{\Emb}{{\mathfrak{Emb}}}

\newtheorem{theorem}{Theorem}[subsection]
\newtheorem{lemma}[theorem]{Lemma}
\newtheorem{proposition}[theorem]{Proposition}

\theoremstyle{definition}
\newtheorem{definition}[theorem]{Definition}

\newtheorem{remark}[theorem]{Remark}

\newtheorem{example}[theorem]{Example}

\begin{document} 

\title{Loops of Legendrians in contact $3$--manifolds.}

\date{\today}


\author{Eduardo Fern\'{a}ndez}
\address{Universidad Complutense de Madrid, 
	Facultad de Matem\'{a}ticas, and Instituto de Ciencias Matem\'{a}ticas CSIC-UAM-UC3M-UCM, C. Nicol\'{a}s Cabrera, 13-15, 28049 Madrid, Spain.}
\email{edufdezfuertes@gmail.com}

\author{Javier Mart\'{i}nez-Aguinaga}
\address{Universidad Complutense de Madrid, 
	Facultad de Matem\'{a}ticas, and Instituto de Ciencias Matem\'{a}ticas CSIC-UAM-UC3M-UCM, C. Nicol\'{a}s Cabrera, 13-15, 28049 Madrid, Spain.}
\email{xabi.martinez.aguinaga@gmail.com}

\author{Francisco Presas}
\address{Instituto de Ciencias Matem\'{a}ticas CSIC-UAM-UC3M-UCM, C. Nicol\'{a}s Cabrera, 13-15, 28049 Madrid, Spain.}
\email{fpresas@icmat.es}

\begin{abstract}
We study homotopically non--trivial spheres of Legendrians in the standard contact $\R^3$ and $\NS^3$. We prove that there is a homotopy injection of the contactomorphism group of $\NS^3$ into some connected components of the space of Legendrians induced by the natural action. We also provide examples of loops of Legendrians that are non-trivial in the space of formal Legendrians, and thus non--trivial as loops of Legendrians, but which are trivial as loops of smooth embeddings for all the smooth knot types.
\end{abstract}

\maketitle

\section{Introduction.}
This small note summarizes the current understanding of the topology of the spaces of Legendrian embeddings in the $3$--dimensional case just using classical invariants. The study of the connected components of those spaces has been a classical topic in the Contact Topology literature \cite{Chekanov, DingGeiges, EliashbergFraser, Etnyre, EtnyreHonda, Lenhard, Szabo}. 

Here we focus on the higher dimensional homotopy groups, in particular in the fundamental group. We define invariants and provide partial classifications by considering the space of Legendrian embeddings as a subspace of the space of formal Legendrian embeddings. Then, we provide non--trivial homotopy elements considered as homotopy classes in the space  of formal Legendrian embeddings and clearly we conclude that the classes are non--trivial in the space of Legendrian embeddings. Do note that, in dimension $3$, there are not known examples of non--trivial spheres of Legendrian embeddings that are trivial as spheres of formal Legendrian embeddings, see \cite{FMP}. 

\textbf{Acknowledgements:} We want to specially thank Alberto Ibort for his support of our research in the last few years. In particular he has been an excellent listener of the study of the topology of the spaces of horizontal embeddings, of which the Legendrian spaces are particular instances.

The authors are supported by the Spanish Research Projects SEV--2015--0554, MTM2016--79400--P, and MTM2015--72876--EXP. The first author is supported by a Beca de Personal Investigador en Formaci\'{o}n UCM.  The second author is funded by Programa Predoctoral de Formaci\'{o}n de Personal Investigador No Doctor del Departamento de Educaci\'{o}n del Gobierno Vasco.

\section{Preliminaries on contact $3$--manifolds.}

We introduce the basic terminology in $3$--dimensional contact topology that we are going to use. For further details see \cite{GeigesBook}.

\subsection{Contact $3$--manifolds.}

\begin{definition}
	Let $M$ be a $3$--manifold. A plane field $\xi\subseteq TM$ is said to be a \em contact distribution \em if it is everywhere non integrable, i.e. locally $\xi$ can be regarded as the kernel of a $1$--form $\alpha\in\Omega^1(M)$ such that 
	\begin{equation}\label{eq:ContactCondition}
	\alpha \wedge d\alpha\neq0.
	\end{equation}
	The pair $(M,\xi)$ is a \em contact $3$--manifold\em. 
\end{definition}

Assume that $\xi$ is \em coorientable. \em Thus, $\xi=\ker \alpha$ for some $1$--form $\alpha\in\Omega^1(M)$ satisfying (\ref{eq:ContactCondition}). The election of $\alpha$ is unique up to a conformal factor. 
From now on, we will focus in the case of \em cooriented contact structures.\em

A diffeomorphism $f:(M_1,\xi_1)\rightarrow(M_2,\xi_2)$ between two contact $3$--manifolds, such that $f_*\xi_1=\xi_2$ is said to be a \em contactomorphism. \em  Denote by $\Cont(M,\xi)$ the group of contactmorphisms from $(M,\xi)$ to itself. Locally any two contact $3$--manifolds are contactomorphic, this is the content of \em Darboux's Theorem \em (see \cite{GeigesBook}, Theorem $2.5.1$).

The \em Reeb vector field \em associated to a contact form $\alpha$ defining a contact manifold $(M,\xi=\ker \alpha)$ is the vector field $R_\alpha$ defined by the conditions $\alpha(R_\alpha)\equiv 1$ and $i_{R_\alpha}d\alpha\equiv0$.

The two basic examples that we are going to study in this note are the following
\begin{example}
	\begin{itemize}
		\item [(i)] The standard contact structure on $\R^3(x,y,z)$ given by $\xi_\std=\ker(dz-ydx)$.
		\item [(ii)] The standard contact structure on $\NS^3\subseteq \C^2(z_1,z_2)$ given by $\xi_\std =T\NS^3 \cap iT\NS^3=\ker(\frac{i}{2}\sum_j z_jd\bar{z}_j-\bar{z}_j dz_j)$, where $i:T\C^2\rightarrow T\C^2$ denotes the standard complex structure on $\C^2$. It follows that $(\R^3,\xi_\std)$ is contactomorphic to $(\NS^3\backslash\{p\},\xi_{\std_{|\NS^3\backslash\{p\}}})$, for any point $p\in\NS^3$ (see \cite{GeigesBook}, Proposition $2.1.8$). This justifies the notation.
	\end{itemize}
\end{example}

\section{Legendrian submanifolds.}

Following \cite{Etnyre,GeigesBook} we introduce the notion of \em Legendrian submanifold \em of a contact manifold. We define the (formal) \em classical invariants. \em We also introduce two (formal) invariants for loops of Legendrians. 

\subsection{Legendrian submanifolds.}

\begin{definition}
	Let $(M,\xi)$ be a contact $3$--manifold. An embedded oriented circle $L\subseteq M$ is said to be \em Legendrian \em if $TL\subseteq\xi$. A \em Legendrian embedding \em is any embedding that parametrizes a Legendrian submanifold. 
\end{definition}

Denote by $\widehat{\Leg}(M,\xi)$ the space of Legendrian submanifolds of $(M,\xi)$ and by $\Leg(M,\xi)$ the space of Legendrian embeddings of $(M,\xi)$. Note that $\widehat{\Leg}(M,\xi)=\Leg(M,\xi)/\Diff^+(\NS^1)$.

A key result in the theory of Legendrian submanifolds is  the \em Weinstein's Tubular Neighbourhood Theorem \em (see \cite{GeigesBook}, Corollary $2.5.9$) which asserts that two diffeomorphic Legendrians have contactomorphic neighbourhoods. Thus, any Legendrian circle $L$ in a contact $3$--manifold has a tubular neighbourhood contactomorphic to a tubular neighbourhood of $\NS^1\times\{0\}\subseteq(\NS^1\times\R^2(\theta,(x,y)),\ker(\cos\theta dx-\sin\theta dy))$.

\subsubsection{Projections.}
There are two distinguished projections $\pi:\R^3\to\R^2$ that are useful in the study of Legendrians. When projecting onto $\R^2$ via the two projections that we will define, each embedding is mapped to a unique curve in $\R^2$. Nevertheless, the converse results are partially true. We can recover a unique Legendrian curve in $(\R^3,\xi_\std)$ from curves in $\R^2$ satisfying certain conditions.

\begin{definition}
We define the \em Lagrangian projection \em as 
	\begin{center}
		$\begin{array}{rccl}
		\pi_{L}\colon & \R^3 & \longrightarrow & \R^2\\
		& (x,y,z)& \longmapsto & (x,y).
		\end{array}$
	\end{center}
\end{definition}
This projection has the property that maps immersed Legendrian curves in $(\R^3,\xi_\std)$ to immersed curves in $\R^2$. In addition, the $z$--coordinate can be recovered by integration:
\[z(t_1)=z(t_0)+\int_{t_0}^{t_1}y(s)x'(s)ds.\] 
Thus, in order for a closed curve in $\R^2$ to lift to a closed Legendrian, it is necessary that the closed immersed disk that bounds the curve has zero (signed) area.

On the other hand we have:

\begin{definition}
We define the  \em front projection \em as 
	\begin{center}
		$\begin{array}{rccl}
		\pi_{F}\colon & \R^3 & \longrightarrow & \R^2\\
		& (x,y,z)& \longmapsto & (x,z).
		\end{array}$
	\end{center}
\end{definition}
We can recover the $y$--coordinate by differentiating:
\[y(s)=\frac{z'(s)}{x'(s)}.\]
The cases $x'(s_0)=z'(s_0)=0$ are allowed as soon as the limit is well defined.

\subsection{Classical invariants.}

There are three \em classical invariants \em of Legendrian embeddings that we will introduce (for simplicity) only in the context of $(\R^3,\xi_\std)$ and $(\NS^3,\xi_\std)$. The first one is the \em smooth knot type of the embedding\em, which is purely topological.

Let $(M,\xi)$ be $(\R^3,\xi_\std)$ or $(\NS^3,\xi_\std)$. Let $\gamma\in\Leg(M,\xi)$ be a Legendrian embedding.
We call \em contact framing \em to the trivialization of its normal bundle $\nu(\gamma)$ given by the Reeb vector field along the knot. We call \em topological framing \em of $\gamma$ to the framing $\F_{\Top}$ of $\nu(\gamma)$ defined by any Seifert surface of $\gamma$.

\begin{definition}
The \em Thurston-Bennequin \em invariant of $\gamma$, denoted by $\tb(\gamma)$, is the twisting of the contact framing with respect to the topological framing. 
\end{definition}

Fix a global trivialization of $\xi$, this election is unique up to homotopy since $\pi_0(\Maps(M,\NS^1))=0$ for the two particular manifolds that we are  studying. Thus, the derivative of the Legendrian embedding defines a map $\gamma':\NS^1\rightarrow\R^2\backslash\{0\}$.

\begin{definition}
The \em rotation number \em of $\gamma\in\Leg(M,\xi)$ is \[\Rot(\gamma)=\deg\gamma'.\]
\end{definition}

It follows that the rotation number is well defined and independent of the trivialization of $\xi$. In the case that $\xi$ is non--trivial, we can define the rotation number for null-homologous knots by just taking a Seifert surface $\Sigma$ on a fixed homology class and selecting a framing for $\xi_{|\Sigma}$  that induces a framing over the boundary by restriction. The induced framing on $\xi_{|\gamma}$  is independent of the choice (see \cite{GeigesBook}, Proposition $3.5.15$).
 
There is an important result relating the three classical invariants of Legendrian embeddings in $(\R^3,\xi_\std)$ and $(\NS^3,\xi_\std)$.

\begin{proposition}[Bennequin's inequality, \cite{Bennequin}]\label{prop:Bennequin}
	Let $\chi(\Sigma)$ denote the Euler characteristic of a Seifert surface $\Sigma$ for the Legendrian embedding $\gamma$. Then the following inequality holds:
\begin{equation}\label{eq:Bennequin}
\tb(\gamma)+|\Rot(\gamma)|\leq -\chi(\Sigma).
\end{equation}
\end{proposition}
\subsection{Invariants for loops of Legendrian embeddings.}
 We consider parametrized loops, when we quotient by the parameter we explicitely mention it. Let $(M,\xi)$ be  $(\R^3,\xi_\std)$ or $(\NS^3,\xi_\std)$.
 We can also define certain (formal) invariants for loops of Legendrian embeddings $\gamma^\theta$ in $\Leg(M,\xi)$. The first invariant is the \em homotopy class of the loop \em of smooth embeddings, i.e. $[\gamma^\theta]\in\pi_1\left(\Emb(\NS^1,M)\right)$, where $\Emb(N_1,N_2)$ denotes the space of embeddings of a manifold $N_1$ into another manifold $N_2$.
 
 The second invariant for loops of Legendrians embeddings is the following
\begin{definition}
The \em rotation number of the loop \em $\gamma^\theta$ is
\[\Rot_{\pi_1}(\gamma^\theta)=\deg(\theta\mapsto(\gamma^\theta)'(0)).\]
\end{definition}

In order to define a different invariant, we assume that there exists a loop of Seifert surfaces $\Sigma^\theta$ for $\gamma^\theta$. Thus, we have a \em loop \em of topological framings for $\nu(\gamma^\theta)$. By means of the Reeb vector field, we can understand this loop as a $\F^{\theta}_{\Top}:\NS^1\rightarrow\R^2\backslash\{0\}$.

\begin{definition}
	The \em Thurston-Bennequin number of the loop \em $\gamma^\theta$ is 
	\[ \tb_{\pi_1} (\gamma^\theta)=-\deg(\theta\rightarrow\F^{\theta}_{\Top}(0)).\] 
\end{definition}

The Thurston-Bennequin number is not necessarily well defined for all the loops of Legendrian embeddings, since we are assuming that we have a loop of topological framings. The key point is the existence and uniqueness of such a loop. 

Assume that a loop of smooth embeddings $\gamma^\theta$ is homotopically trivial, i.e. there is $\gamma^z$, $z\in\D$, whose boundary is $\gamma^\theta$. Thus, there is a unique topological framing on $\gamma^z_{|z=0}$ induced by any choice of Seifert surface. This induces a unique $1$--parametric family of topological framings on the loop $\gamma^\theta$. Now, assume that we have a $2$-sphere $S$ of smooth embeddings. If we understand it as the union of two disks we get by the previous construction two possibly different topological framings for the equator $\gamma^\theta$. The difference between the two loops is measured by an integer $d(S)$. We conclude that we have a well defined morphism of abelian groups $d: \pi_2(\Emb(\NS^1, M)) \rightarrow \Z;S\mapsto d(S)$. Obviously $\im (d)=k_0\Z$ for some integer $k_0$.

Recall that there is a natural left action of $\Diff^+(M)$ on $\Emb(\NS^1,M)$. Denote by $\Phi_\gamma:\Diff^+(M)\rightarrow\Emb(\NS^1,M)$ the orbit of an embedding $\gamma$ under the action.

We can state the following proposition that is obviously true by the previous discussion. 

\begin{proposition}\label{prop:TBWellDefined}
	Let $\gamma^\theta$ be a loop Legendrian embeddings. In any  of the two following cases $\tb_{\pi_1}$ is a well defined invariant:
	\begin{enumerate}
		\item [(A)] If $\pi_1(\Phi_\gamma)$ is an monomorphism and $[\gamma^\theta]\in\im(\pi_1(\Phi_\gamma))$ or
		\item [(B)] $\gamma^\theta$ is trivial as a loop of smooth embeddings. In this case $\tb_{\pi_1}(\gamma^\theta)\in\Z/k_0\Z$.
	\end{enumerate}
\end{proposition}

\begin{remark}
	Assume that $\gamma^\theta$ satisfies condition (A) and condition (B). Then, $k_0=0$ and $\tb_{\pi_1}^B(\gamma^\theta)=\tb_{\pi_1}^A(\gamma^\theta)\in\Z$. The reason is that condition (B) provides a capping disk $\gamma^{r,\theta}$ for $\gamma^\theta$. Now the segment $\gamma^{r,\theta_0}$, by the \em Isotopy Extension Theorem, \em can be represented by a segment of diffeomorphisms $\Psi_{r,\theta_0}$, i.e. $\Psi_{0,\theta_0}=\Id$ and $\Psi_{r,\theta_0}(\gamma^0)=\gamma^{r,\theta_0}$. Thus, $\Psi_{1,\theta}(\gamma^0)=\gamma^\theta$ and the two definitions do coincide. This also implies $k_0=0$. For this reason, we do not specify if we are considering the type $(A)$ invariant or the type $(B)$ invariant in the discussion that follows.
\end{remark}

Finally, we can state the following useful formulas
\begin{proposition}\label{prop:FormulaParametrization}
	Let $\gamma^\theta$ be a loop of Legendrian embeddings. For each $k\in\Z$ define the reparametrizations  $\gamma^{\theta,k}(t):=\gamma^\theta(t-k\theta)$. Then,
	\begin{itemize}
		\item [(i)] $\Rot_{\pi_1} (\gamma^{\theta,k})=\Rot_{\pi_1}(\gamma^\theta)-k\Rot(\gamma^0)$,
		\item[(ii)] $\tb_{\pi_1} (\gamma^{\theta,k})=\tb_{\pi_1}(\gamma^\theta)-k\tb(\gamma^0)$, whenever the $\tb_{\pi_1}$ is well defined for both loops.
	\end{itemize}
\end{proposition}

\begin{remark}
	Hatcher's work about knot spaces in $\NS^3$ implies that $\tb_{\pi_1}$ is well defined for many cases:
	\begin{itemize}
		\item The connected component $\Emb_0 (\NS^1,\NS^3)$ of the trivial embedding has the homotopy type of $V_{4,2}$, the space of parametrized great circles (see \cite{HatcherSmale}, Appendix: Equivalence $(6)$). Note that, $\pi_1 (V_{4,2})=0$ and $\pi_2 (V_{4,2})\cong \Z$. In this case $\tb_{\pi_1}\in\Z$ is always defined since $k_0=0$. Indeed, consider $\NS^3$ as a submanifold of the quarternions $\R^4(i,j,k)$, the generator of $\pi_2(\Emb_0 (\NS^1,\NS^3))$ is the $2$--sphere $\{\gamma_p:p\in\NS^2(i,j,k)\}$, meaning that $\gamma_p(\theta) = \cos \theta + p \sin \theta$. It is clear that a normal framing for $\gamma_i$ is given by $\tau_i(\theta)= \langle j, k \rangle$. We choose as equator the loop of curves $\gamma_p$ with $p\in \NS^1(j,k)$. We look for a family of diffeomorphisms $A_{p,r}: \R^4 \to \R^4$, $(p,r) \in \NS^1(j,k) \times [0,1]$, such that $A_{p,r}(\gamma_i)= \gamma_{(\cos \frac\pi2 r)  i + (\sin \frac\pi2 r) p}$ and $A_{p,r}(1)=1$. We can, indeed, choose $A_{p,r}$ to be the linear rotation of angle $\frac\pi2 r$ around the axis defined by $\langle 1, i\cdot p \rangle$. With this choice it is clear that the $\NS^1$-parametric  family of knots $\gamma_p= A_{p,1} \gamma_i$ has an associated family of framings $\tau_p= \langle (\cos \phi) i + \sin \phi ((\sin \phi) j -(\cos \phi)k),  (\sin \phi) i + \cos \phi ((\sin \phi) j -(\cos \phi) k) \rangle$, with $p = (\cos \phi) j +( \sin \phi) k$. Choosing the south pole $-i \in \NS(i,j,k)$, we define the framing $\tau_{-i}(\theta)= \langle -j, -k \rangle$. We repeat the previous process but taking a family of linear isomorphisms preserving the axis $\langle 1, -i\cdot p \rangle$. Obviously, we get a $\NS^1$--family of framings $\tilde \tau_p$ that satisfies $\tilde \tau_p = \tau_p$. We have proven that $d([\gamma_p])=0$.
		
		\item The connected component $\Emb_{p,q}(\NS^1,\NS^3)$ of a non--trivial parametrized  $(p,q)$ torus knot has the homotopy type of $\SO(4)$ and the homotopy equivalence is induced by the action $\Diff^+(\NS^3)\rightarrow \Emb_{p,q}(\NS^1,\NS^3)$ (see \cite{HatcherKnots}, Theorem $1$). Recall that $\Diff^+ (\NS^3)$ is homotopy equivalent to $\SO(4)$ (see \cite{HatcherSmale}). In this case, $\tb_{\pi_1}\in\Z$ is always defined.
		
		\item The connected component $\Emb'(\NS^1,\NS^3)$ of an hyperbolic parametrized knot $\gamma$ has the homotopy type of $\NS^1\times\SO(4)$ (see \cite{HatcherKnots}, Theorem $1$). In particular, $\Emb'(\NS^1,\NS^3)$ has trivial second homotopy group. Moreover, the map $\pi_1(\Phi_\gamma)$ is injective. Thus, $\tb_{\pi_1}\in \Z$ is well defined for loops in $\im(\Phi_\gamma)$ and for smoothly trivial loops.
	\end{itemize}
\end{remark}

\section{The formal viewpoint.}
Following \cite{FMP} we study Legendrians from a formal viewpoint in $(\R^3,\xi_\std)$ and $(\NS^3,\xi_\std)$. The assumption of restricting to the standard structures is only made for simplicity in the statements\footnote{In fact, the statements follow for any contact structure on $\R^3$ or $\NS^3$.}. For a more general discussion see \cite{Murphy}.

\subsection{Formal Legendrian Embeddings.}

\begin{definition}
	Let $(M,\xi)$ be $(\R^3,\xi_\std)$ or $(\NS^3,\xi_\std)$. A \em formal Legendrian embedding \em into $(M,\xi)$ is a pair $(\gamma,F_s)$ such that
	\begin{itemize}
		\item [(i)] $\gamma:\NS^1\rightarrow M$ is an embedding,
		\item [(ii)] $F_s: \NS^1\rightarrow \gamma^* (TM\backslash\{0\})$, $s\in[0,1]$, is a homotopy between $F_0=\gamma'$ and $F_1:\NS^1\rightarrow\gamma^* (\xi\backslash\{0\})\subseteq\gamma^*(TM\backslash\{0\})$.
	\end{itemize}
\end{definition}

Trivialize $TM$ and $\xi$. From now on we understand $F_s:\NS^1\rightarrow \NS^2$ and $F_1:\NS^1\rightarrow\NS^1=\NS^2\cap\xi$. On $(\R^3(x,y,z),\xi_\std)$ we fix the framing $\xi_\std=\langle \partial_x+y\partial_z,\partial_y\rangle$. On $(\NS^3,\xi_\std)$ we fix the framing given by $\xi_\std (p)=\langle jp,kp\rangle$, where we are using quarternionic notation, i.e. $\NS^3\subseteq \R^4(i,j,k)$.

Denote by $\FLeg(M,\xi)$ the space of formal Legendrian embeddings. In order to study the homotopy type of the space of formal Legendrians we introduce the following auxiliary space $\FLegO(M,\xi)=\Emb(\NS^1,M)\times L\NS^1$, where $LX$ denotes the free loop space of a connected manifold $X$. Recall that $LX$ has the homotopy type of $X\rtimes\Omega_p (X)$ and, moreover, if $X$ is a Lie group then $LX\cong X\times\Omega_1 (X)$. We have a natural fibration 
\begin{center}
	$\begin{array}{rccl}
	f:\FLeg(M,\xi)&\longrightarrow &  \FLegO(M,\xi)\\
	(\gamma,F_s)& \longmapsto & (\gamma,F_1).
	\end{array}$
\end{center}
The morphism $f$ is surjective. Indeed, given $(\gamma,F_1)$ we need to find a homotopy between $\gamma'$ and $F_1$ inside $L\NS^2$. Since $F_1$ is null homotopic (in $L\NS^2$) by the Legendrian condition, this is equivalent to say that $\gamma'$ is null homotopic which is true, by dimensional reasons, for every embedding $\gamma\in\Emb(\NS^1,M)$. Fix as base point $(\gamma,\gamma')\in\FLegO(M,\xi)$, with $\gamma\in\Leg(M,\xi)$. The fiber over this point is $\Omega_{\gamma'}(L\NS^2)$. Denote the diagonal maps in the associated long exact sequence in homotopy by $\partial_k: \pi_k(\FLegO(M,\xi))\rightarrow\pi_{k-1}(\Omega_{\gamma'}(L\NS^2))\cong\pi_k (L\NS^2)$. Recall that, by the Smale-Hirsch $h$--principle for immersions (\cite{Hirsch}), $\Imm(\NS^1,M)$ has the homotopy type of $LM\times L\NS^2$. Let $p_2:LM\times L\NS^2\rightarrow L\NS^2$ be the projection onto the second factor and $i:\Emb(\NS^1,M)\hookrightarrow\Imm(\NS^1,M)$ the natural inclusion. We have the following
\begin{lemma}\label{lem:DiagonalMaps}
	The homorphisms $\partial_k$ and $\pi_k (p_2)\circ \pi_k (i)$ coincide. More precisely, if $(\gamma^z,F_{1}^{z})\in\pi_k(\FLegO(M,\xi))$ then 
	\[ \partial_k (\gamma^z,F_{1}^{z})=\pi_k (p_2)\circ \pi_k (i) (\gamma^z)\in\pi_k (L\NS^2). \]
\end{lemma}
\begin{proof}
 The image of $(\gamma^z,F_{1}^{z})$ by $\partial_k$ measures the difference between the derivative $(\gamma^z)'$ and $F_{1}^{z}$ as elements in $\pi_k (L\NS^2)$. The homotopy class of $F_{1}^{z}$ is zero by the Legendrian condition. Thus,  $\partial_k (\gamma^z,F_{1}^{z})=(\gamma^z)'\in\pi_k (L\NS^2)$, i.e. $\partial_k (\gamma^z,F_{1}^{z})=\pi_k (p_2)\circ \pi_k (i) (\gamma^z)$.
\end{proof}

With the previous discussion in mind one can conclude the following well known fact, we refer the reader to \cite{FMP} or \cite{Murphy}  for a proof,

\begin{theorem}\label{thm:ClassificationFLeg}
	Formal Legendrian embeddings are classified by their parametrized knot type, rotation number and Thurston--Bennequin invariant.
\end{theorem}

About the fundamental group of the space of formal Legendrian embeddings we can state the following

\begin{theorem}[\cite{FMP},Theorem $3.4.1$]\label{thm:FundamentalGroupFLeg}
	Let $\gamma\in\Leg(M,\xi)$ be a Legendrian embedding. Fix $(\gamma,\gamma')\in\FLeg(M,\xi)$ as the base point. Then, there exists a number $m\in\Z_{\geq0}$, depending only on the parametrized knot type of $\gamma$, such that the following sequence
		\begin{displaymath}
	\xymatrix@M=10pt{
		0\ar[r] & \Z\oplus\Z_m\ar[r] & \pi_1 (\FLeg(M,\xi_{std})) \ar[r] & \pi_1 (\Emb(\NS^1,M))\oplus\Z\ar[r] & 0 }
	\end{displaymath}
	is exact. Moreover, the last $\Z$ factor corresponds to the $\Rot_{\pi_1}$  invariant.
\end{theorem}

\section{The action of $\Cont(\NS^3,\xi_\std)$ on the space $\widehat{\Leg}(\NS^{3},\xi_\std)$.}

\subsection{The action of the contactomorphism group on the space of Legendrians.}

Recall that on $\NS^3\subseteq\C^2$ the standard contact structure $\xi_\std$ is defined as the complex tangencies of the $3$--sphere; i.e. $\xi_{\std}(p)=T_p \NS^{3} \cap i(T_p \NS^{3})$, $p\in\NS^{3}$. Thus we have a natural inclusion 
\begin{equation}\label{eq:InclusionU(2)}
\U(2)\hookrightarrow\Cont(\NS^{3},\xi_\std).
\end{equation}
This map has a geometrically left inverse given by the evaluation of the $1$--jet of a contactomorphism at the north pole $N\in \NS^{3}$. Hence, the last inclusion induces an injection in all homotopy groups. Moreover, it is a well known fact (\cite{Eliashberg}) that the inclusion $\U(2)\hookrightarrow \Cont(\NS^3,\xi_\std)$ is a weak homotopy equivalence (see \cite{CS} for a complete proof). 

Consider the restriction of the natural action of the contactomorphism group on the space of Legendrians to the unitary group. Fix a Legendrian $L \in\widehat{\Leg}(\NS^{3},\xi_\std)$. The orbit of this Legendrian is described by the map
\begin{equation}\label{eq:OrbitLegendrian}
\begin{array}{rccl}
\hat{\Phi}_L:&\U(2)&\longrightarrow &  \widehat{\Leg}(\NS^{3},\xi_\std)\\
&A& \longmapsto & A(L).
\end{array}
\end{equation}
Observe that we have an analogous action in the space of Legendrian embeddings. The orbit of $\gamma\in\Leg(\NS^{3},\xi_\std)$ is given by 
\begin{equation}\label{eq:OrbitLegendrianEmbedding}
\begin{array}{rccl}
\Phi_\gamma:&\U(2)&\longrightarrow &  \Leg(\NS^{3},\xi_\std)\\
&A& \longmapsto & A\cdot\gamma.
\end{array}
\end{equation}

\subsection{Homotopy injection of $\Cont(\NS^3,\xi_\std)$ in $\widehat{\Leg}(\NS^3,\xi_\std)$.}

The main result of this section is the following one
\begin{theorem}\label{thm:InjectionUnitaryGroup}
	The map \[ \pi_k (\hat{\Phi}_L):\pi_k (\U(2),\Id)\rightarrow\pi_k (\widehat{\Leg}(\NS^{3},\xi_\std),L) \] is an injection for all $k\geq 2$. Moreover, if one of the following conditions is satisfied
	\begin{itemize}
		\item $\Rot(L)=0$ or
		\item $\tb(L)\neq 0$ and $L$ is an unknot or a torus knot or a hyperbolic knot,
	\end{itemize}
	then the map $ \pi_1 (\hat{\Phi}_L)$ is also an injection.
\end{theorem}
\begin{remark}
	By Proposition \ref{prop:Bennequin} the second condition is always satisfied for any Legendrian unknot. Moreover, the classification of the Legendrian figure eight knots (see \cite{EtnyreHonda}) implies that the second condition is always satisfied.
\end{remark}

\begin{proof}
For $p\in\NS^3$, we have that $\xi_\std (p)=\langle jp, kp\rangle$. Let $\gamma\in\Leg(\NS^3,\xi_\std)$ be any parametrization of $L$. Since the unitary groups acts transitively on $\NS^3$ we may assume that $\gamma(0)=(1,0)^t\in\NS^3\subseteq\C^2$.

\begin{lemma}
	The maps \[ \pi_k (\Phi_\gamma):\pi_k (\U(2),\Id)\rightarrow\pi_k (\Leg(\NS^{3},\xi_\std),\gamma)\] are injective for all $k$.
\end{lemma}
\begin{proof}
The composition 
	\begin{center}
		$\begin{array}{rcccl}
		\SU(2) &\longrightarrow &  \Leg(\NS^3,\xi_\std) & \longrightarrow & \NS^3 \\
		A & \longmapsto & A\cdot\gamma & \longmapsto & A(\gamma(0)),
		\end{array}$
	\end{center}
defines a diffeomorphism. Thus, since $\pi_k(\U(2))\cong\pi_k (\SU(2))$ for $k\geq2$, we conclude that $\pi_k (\Phi_\gamma)$ is injective for $k\geq 2$.

Observe that $\pi_1 (\U(2))=\{[A_{\theta}]^m=[A_{\theta}^{m}]:m\in\Z\}$ where
\[ A_\theta =\left( {\begin{array}{cc}
		1  & 0 \\
		0 & e^{ i \theta} \\
\end{array} } \right), \theta\in\NS^1. \]
Hence, \[\Rot_{\pi_1}(A_{\theta}^{m} \cdot \gamma)=m\] and $\pi_1 (\Phi_\gamma)$ is also injective. 
\end{proof}

Since $\widehat{\Leg}(\NS^3,\xi_\std)=\Leg(\NS^3,\xi_\std)/\Diff^+(\NS^1)$ the last lemma implies that $\pi_k (\hat{\Phi}_L)$ is injective for $k\geq 2$.

To conclude the injectivity of the map $\pi_1(\hat{\Phi}_L)$ we must check that the loops $A_{\theta}^{m}(L)$, $m\neq 0$, are non trivial for any choice of parametrization. The possible parametrizations of these loops are given by \[ \gamma^{\theta,k}_{m}(t)=A_{\theta}^{m}\gamma(t-k\theta), k\in\Z.\]

Assume that there exists $k\in\Z$ such that $\gamma^{\theta,k}_{m}$ is trivial. Thus,
\begin{equation}\label{eq:RotReparametrization}
\Rot_{\pi_1}(\gamma^{\theta,k}_{m})=\Rot_{\pi_1}(\gamma^{\theta,0}_{m})-k\Rot(L)=m-k\Rot(L)=0.
\end{equation} 
Hence, if $\Rot(L)=0$ the map $\pi_1 (\hat{\Phi}_L)$ is injective.

From now on assume that $L$ is an unknot or a torus knot or a hyperbolic knot. The $\tb_{\pi_1}$ invariant is always well defined  for the unknot and any torus knot. Moreover, for a hyperbolic knot $\tb_{\pi_1}$ is well defined in $\im(\hat{\Phi}_L)$. Assume that there exists $k\in\Z$ such that $\gamma^{\theta,k}_m$ is trivial for $m\neq0$. Thus, $k\neq0$ by (\ref{eq:RotReparametrization}). On the other hand
\begin{equation}\label{eq:TbReparametrization}
\tb_{\pi_1}(\gamma^{\theta,k}_{m})=\tb_{\pi_1}(\gamma^{\theta,0}_{m})-k\tb(L)=-k\tb(L)=0.
\end{equation} 
Realize that, since $\gamma^{\theta,0}_m= A_{\theta}^m \cdot \gamma^0$. Thus, after fixing a Seifert surface $S_0$ for $\gamma^0$, we get a family of Seifert surfaces $S_{\theta}$ and it is clear that $ \tb_{\pi_1}(\gamma^{\theta,0}_{m})=0$. Hence, if $\tb(L)\neq0$ the map $\pi_1 (\hat{\Phi}_L)$ is an injection.
\end{proof}

\subsection{Non injectivity of the map $\pi_1 (\widehat{\Leg}(\NS^3,\xi_\std),L)\rightarrow \pi_1(\Knots(\NS^3),L)$.}
Let $\Knots(\NS^3)$ be the space of embedded circles in $\NS^3$. The map (\ref{eq:OrbitLegendrian}) allows us to construct plenty of examples of loops which are homotopically trivial in the smooth category, i.e. inside $\Knots(\NS^3)$, but non--trivial in the Legendrian setting.

\begin{proposition}\label{cor:NonInjectiveGeneral}
Let $L\in\widehat{\Leg}(\NS^3,\xi_\std)$ be a Legendrian which satisfies one of the following conditions:
\begin{itemize}
	\item $|\Rot(L)|\neq1,2$ or
	\item $\tb(L)\neq 0$ and $L$ is an unknot or a torus knot or a hyperbolic knot.
\end{itemize}
Then, the homomorphism $\pi_1 (\widehat{\Leg}(\NS^3,\xi_\std), L)\rightarrow \pi_1(\Knots(\NS^3), L)$, induced by the inclusion, is non injective.
\end{proposition}
\begin{figure}[h]
	\centering
	\includegraphics[width=0.8\textwidth]{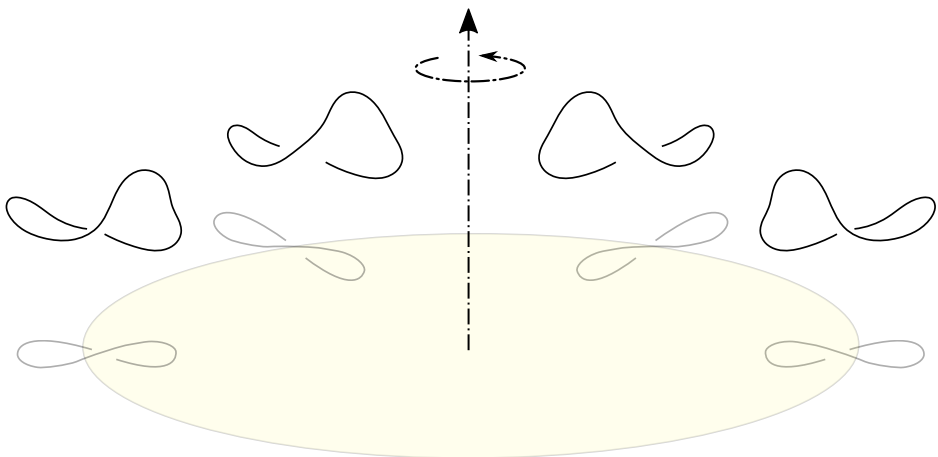}
	\caption{Schematic picture of the loop $A^m_\theta(L)$ for the standard Legendrian unknot $L$, $\theta\in[0,\pi/m]$. \label{fig:LoopSmooth}}
\end{figure}
\begin{proof}
Assume that the second condition holds. Let $m\neq 0$ be any even integer, by Theorem \ref{thm:InjectionUnitaryGroup} the loop $A_{\theta}^m(L)$ is non--trivial. Finally, observe that since $m$ is even $A_{\theta}^m$ is trivial as a loop in $\SO(4)$. Thus, $A_{\theta}^m(L)$ is homotopically trivial inside $\Knots(\NS^3)$.

On the other hand, assume that the first condition holds. Let $\gamma\in\Leg(\NS^3,\xi_\std)$ be any parametrization of $L$. Take $m=\Rot(L)+2$ if $\Rot(L)$ is even and $m=\Rot(L)+1$ if it is odd. Since $m$ is even $A_{\theta}^{m} (L)$ is trivial in $\Knots(\NS^3)$. However, the loop $A_{\theta}^{m} (L)$ is non-trivial inside $\widehat{\Leg}(\NS^3,\xi_\std)$. Indeed, all the parametrizations of the loop are given by $\gamma^{\theta,k}_{m}(t)=A_{\theta}^{m}\gamma(t-k\theta)$. The equality $\Rot_{\pi_1}(\gamma^{\theta,k}_{m})=m-k\Rot(L)=0$ cannot hold for any $k\in\Z$ since it implies that $\Rot(L)$ divides $1$ or $2$ and this is not true.
\end{proof}

\end{document}